\documentclass[10pt,letter]{amsart}

\usepackage{amssymb,latexsym, amsmath, amsxtra, amsthm, mathrsfs, bm}

\usepackage[dvips]{graphics}
\usepackage{xypic,xcolor}
\usepackage{subcaption}
\usepackage{verbatim}

\usepackage[abs]{overpic}
\usepackage{hyperref}
\usepackage{mathtools, tikz}
\usetikzlibrary{positioning, shapes.geometric}
\usepackage{enumitem}

\DeclareMathAlphabet{\mathbbold}{U}{bbold}{m}{n}

\allowdisplaybreaks[4]

\hypersetup{colorlinks=true,linkcolor=blue,citecolor=blue}

\renewcommand{\le}{\leqslant}
\renewcommand{\leq}{\leqslant}
\renewcommand{\ge}{\geqslant}
\renewcommand{\geq}{\geqslant}

\newtheorem{Theorem}{Theorem}[section]
\newtheorem*{Theorem*}{Theorem}
\newtheorem{Lemma}[Theorem]{Lemma}
\newtheorem{Proposition}[Theorem]{Proposition}

\theoremstyle{definition}
\newtheorem{Definition}[Theorem]{Definition}
\theoremstyle{remark}
\newtheorem{Remark}[Theorem]{Remark}
\numberwithin{equation}{section}

\theoremstyle{plain}
\newcommand{\thistheoremname}{}
\newtheorem{genericthm}[Theorem]{\thistheoremname}
	
\newtheorem*{genericthm*}{\thistheoremname}
\newenvironment{namedthm*}[1]
{\renewcommand{\thistheoremname}{#1}%
	\begin{genericthm*}}
	{\end{genericthm*}}
 
\newcommand{\R}{\mathbb{R}}
\newcommand{\B}{\mathbf{B}}
\newcommand{\C}{\mathbb{C}}
\newcommand{\N}{\mathbb{N}}
\newcommand{\D}{\mathbb{D}}

\renewcommand{\H}{\mathbb{H}}

\newcommand{\Q}{\mathbb{Q}}

\newcommand{\diam}{\operatorname{diam}}

\newcommand{\card}{\operatorname{card}}

\renewcommand{\:}{\colon}

\newcommand{\sub}{\subseteq}

\renewcommand{\P}{\mathbb{P}}

\renewcommand{\:}{\colon}

\renewcommand{\=}{\coloneqq}

\newcommand{\SLE}{\mathrm{SLE}}
\newcommand{\CLE}{\mathrm{CLE}}
\newcommand{\LQG}{\mathrm{LQG}}

\begin{document}
\title{Quasisymmetric geometry of low-dimensional random spaces}

\author{Gefei Cai}
\address{Beijing International Center for Mathematical Research, Peking University, No. 5 Yiheyuan Road, Haidian District, Beijing, China, 100871}
\email{caigefei1917@pku.edu.cn}

\author{Wen-Bo Li}
\address{Beijing International Center for Mathematical Research, Peking University, No. 5 Yiheyuan Road, Haidian District, Beijing, China, 100871}
\email{liwenbo@bicmr.pku.edu.cn}

\author{Tim Mesikepp}
\address{Department of Mathematics, Western Washington University, 516 High Street, Bellingham, Washinton, USA, 98225}
\email{mesiket@wwu.edu}


\begin{abstract}
    We initiate a study of the quasisymmetric uniformization of naturally arising random fractals and show that many of them fall outside the realm of quasisymmetric uniformization to simple canonical spaces. We begin with the trace, the graph of Brownian motion, and various variants of the Schramm-Loewner evolution $\SLE_\kappa$ for $\kappa>0$, and show that a.s. neither is a quasiarc. After that, we study the conformal loop ensemble $\CLE_\kappa$, $\kappa \in (\frac{8}{3}, 4]$, and show that the collection of all points outside the loops is a.s. homeomorphic to the standard Sierpi\'nski carpet, but not quasisymmetrically equivalent to a round carpet.
\end{abstract}

\maketitle

\section{Introduction}\label{Section: Introduction}

The uniformization theorem in complex analysis states that every simply connected Riemann surface is conformally equivalent to either the open disk, the complex plane, or the Riemann sphere.  It is natural to seek parallel results in more flexible categories of mappings, such as the \emph{quasisymmetric} mappings.  Recall a homeomorphism $f: X \to Y$ between metric spaces is an \emph{$\eta$-quasisymmetry} if for all $x, y, z \in X$ with $x \neq z$ one has
\begin{align}\label{def:quasisymmetry}
	\frac{d_Y(f(x),f(y))}{d_Y(f(x),f(z))} \leq \eta \left(\frac{d_X(x,y)}{d_X(x,z)}\right),
\end{align}
where $\eta:[0, \infty) \longrightarrow [0, \infty)$ is an increasing continuous function such that  $\eta(0)=0$ and $\eta(t)\underset{t\to\infty}{\longrightarrow}\infty$. A mapping $f$ is called \emph{quasisymmetric} if it is $\eta$-quasisymmetric for some \emph{distortion function} $\eta$.  A quasisymmetry thus distorts relative distances by a bounded amount, see \cite[Proposition~$10.8$]{Hei01}. These maps were introduced by Tukia and V\"ais\"al\"a \cite{TV80} as a generalization of conformal and quasiconformal mappings to the setting of general metric spaces.

It is natural to ask when two metric spaces $X$ and $Y$ are quasisymmetrically equivalent, given that they are homeomorphic to each other. In another point of view, thinking along the lines of the uniformization theorem, we may ask when we can quasisymmetrically uniformize a given metric space $X$ to some canonical target $Y$. We may also ask for a quasisymmetric characterization of $Y$ in the case that a unique uniformization does not exist. This of course depends on the context how the term “standard” is precisely defined. These questions are partially motivated by problems from geometric group theory. See \cite[Section~$3$, $4$ and $5$]{Bon06}.

Given the progress in understanding the quasisymmetric geometry of many deterministic objects, a natural step is to study quasisymmetric geometry in the stochastic world. In this paper we essentially study three random processes, namely, the Brownian motion, the Schramm-Loewner evolution $\SLE_\kappa$ for $\kappa>0$, and the conformal loop ensemble $\CLE_\kappa$ for $\kappa \in (\frac{8}{3}, 4]$.  We ask whether we can quasisymmetrically uniformize them to line segments, in the first two cases, or to a round carpet, in the latter. In each case our results give a negative answer.

In what follows, we provide the necessary background and rigorously state our results. We begin by presenting results for random Cantor sets in Section \ref{Sec:IntroCantor}, with a particular focus on those arising from fractal percolation . Section \ref{Sec:IntroArc} addresses the line uniformization case for Brownian motion and $\SLE_\kappa$, while Section \ref{Sec:IntroCarpet} covers the carpet case. For completeness, we review the established results for the sphere in Section \ref{Sec:IntroSphere}.

\subsection{Quasi-Cantor set}\label{Sec:IntroCantor}

The study of random Cantor sets can be traced back to Mandelbrot \cite{Man74, Man83}. He introduced a statistically self-similar family of random Cantor sets, i.e., the \emph{fractal percolation}. Fix a number $0< p<1$ and define a random set in the following way. We divide $[0,1]^n$ into $l^n$ many disjoint cubes (except at the boundary) whose coordinates are consists of $l$-adic numbers. Each cube is kept with probability $p$, and the kept cubes are further divided, and again each subcube is kept with probability $p$. The resulting set is a fractal percolation and we denote it by $\Lambda(l, p)$. It is well known that if $p \leq n^{-l}$, then $\Lambda(l, p)$ is a.s. empty and if $p > n^{-l}$, then $\Lambda(l, p)$ is either empty or $\dim_H(\Lambda(l, p)) = n + \frac{\log p}{\log l}$ with positive probability \cite[Theorem~$3.7.1$]{BP17}. 

A topological space is homeomorphic to the standard Cantor set if it is perfect, nonempty, compact, metrizable and totally disconnected. It follows \cite{CCD88} that there exists a $p_c \in (0,1)$ such that $\Lambda(l, p)$ is totally disconnected conditioned on that $\Lambda(l, p)$ is nonempty for any $p < p_c$. To prove that $\Lambda(l, p)$ is homeomorphic to the standard Cantor set a.s., we only need to show that every point in $\Lambda(l, p)$ is a limit point a.s.. If it is not, there exists a $l$-adic cube contains only one point in $\Lambda(l, p)$. This only happens with zero probability. Thus a.s. $\Lambda(l, p)$ is homeomorphic to the standard Cantor set conditioned on that $\Lambda(l, p)$ is nonempty for any $p < p_c$. 

A metric space is quasisymmetric to the standard Cantor set if it is compact, doubling, uniformly perfect, and uniformly disconnected \cite[Theorem~$15.11$]{Hei01}. The random structure of fractal percolation implies that it is not uniformly perfect and uniformly disconnected a.s.. For more details, see \cite[Corollay~$4.14$]{BE19} and \cite[Theorem~$6$]{CORS17}. Then it is not quasisymmetric to the standard Cantor set.

\subsection{Quasiarcs}\label{Sec:IntroArc}

Quasisymmetric uniformization problems in one dimension rely on Tukia and V\"ais\"al\"a's characterization of \emph{quasiarcs}.  An \emph{arc} in a metric space is the image of an interval (either open, closed or half open half closed) under a homeomorphism, while a \emph{quasiarc} is the image of an interval under a quasisymmetry\footnote{The ``classical'' definition of a quasicircle is a topological circle $\gamma \sub \C$ such that $\gamma = f(\mathbb{S}^1)$ for some quasiconformal mapping $f$ of the Riemann sphere.  Since planar quasiconformal mappings are, in fact, quasisymmetries, the above definition extends the classical one and allows codomains which are general metric spaces.}. Similarly, a \emph{quasiline}, \emph{quasiray} or a \emph{quasicircle} is the image of $\R$, $[0, \infty)$ or $\mathbb{S}^1$ under a quasisymmetry, respectively.

The study of quasiarcs are motivated by a question of Papasoglu: whether the boundary of a Gromov hyperbolic space contains a quasicircle. This question was complete solved by Bonk and Kleiner, who showed that the boundary of a Gromov hyperbolic group contains a quasicircle if and only if the group is not virtually free \cite[Theorem~$5.4$]{Bon06}.

Tukia and V\"ais\"al\"a's characterization says that quasiarcs are precisely those metric spaces $X$ which are \emph{doubling} and of \emph{bounded turning} \cite[Theorem~$4.9$]{TV80}. We recall $X$ is \emph{$B$-doubling} for a constant $B \geq 1$ if every ball of radius  $r$ in $X$ can be covered by $B$ or fewer balls of radius  $r/2$. Meanwhile, $X$ is of \emph{$C$-bounded turning} for a constant $C \geq 1$ if all points $x,y\in X$  can be joined by a curve $\gamma$ such that $\diam(\gamma) \leq Cd(x,y)$.

Are there natural stochastic arcs which are doubling and of bounded turning?  A first reasonable candidate to consider is Brownian motion, a real-valued continuous stochastic field whose stationary and independent increments are centered Gaussians. See \cite{MP10} for more information.  A second potential candidate is the related Schramm-Loewner Evolution SLE$_\kappa$, a one-parameter family of conformally-invariant random curves which describe the scaling limits of many interfaces at criticality in two-dimensional statistical physics.  See \cite{Law05} for more information. Furthermore, one can consider a natural generalization of $\SLE_\kappa$, namely the $\SLE_\kappa(\rho)$ curves, where we further specify additional information of the location of a marked point in the domain or on its boundary. The $\SLE_\kappa(\rho)$ curves are first introduced in \cite{LSW03} and they will arise naturally in two-dimensional statistical physics models with different types of boundary conditions. Our first two results show neither of these contain quasiarcs.

\begin{Theorem}\label{BM}
    The following almost surely holds: for any $n \in \N$
    \begin{enumerate}
    	\item Any arc on the trace of the $n$-dimensional Brownian motion is not a quasiarc for $n > 1$.
    	
    	\item Any arc on the graph of the $n$-dimensional Brownian motion is not a quasiarc.
    \end{enumerate} 
\end{Theorem}

\begin{Theorem}\label{SLE}
    The following almost surely holds: for any $\kappa >0$ and $\rho \in \R$,
    \begin{enumerate}
    \item the chordal $\SLE_\kappa$ and $\SLE_\kappa(\rho)$ curve in $\H$ is not a quasiray.
    
    \item the radial $\SLE_\kappa$ and $\SLE_\kappa(\rho)$ curve in $\D$ is not a quasiarc.

    \item the whole-plane $\SLE_\kappa$ and $\SLE_\kappa(\rho)$ curve from $0$ to $\infty$ is not a quasiray.
    \end{enumerate}
\end{Theorem}

While these results may not come as a surprise to experts, we believe it is important to include them in the literature for the sake of completeness. Moreover, the proof involves some challenges, and we introduce new ideas to overcome them.

\subsection{Quasiround carpets}\label{Sec:IntroCarpet}

A metric space is a \emph{metric (Sierpi\'nski) carpet} if it is homeomorphic to the standard Sierpi\'nski carpet $\mathbb{S}_{1/3}$, as in Figure \ref{standardcarpet}.  A \emph{round carpet} is a metric carpet in $\C$ where each complimentary component is a (round) disk.  A \emph{quasiround carpet} is the quasisymmetric image of a round carpet.

\begin{figure}[htbp]
	\begin{center}
		\includegraphics[height = 1.5in]{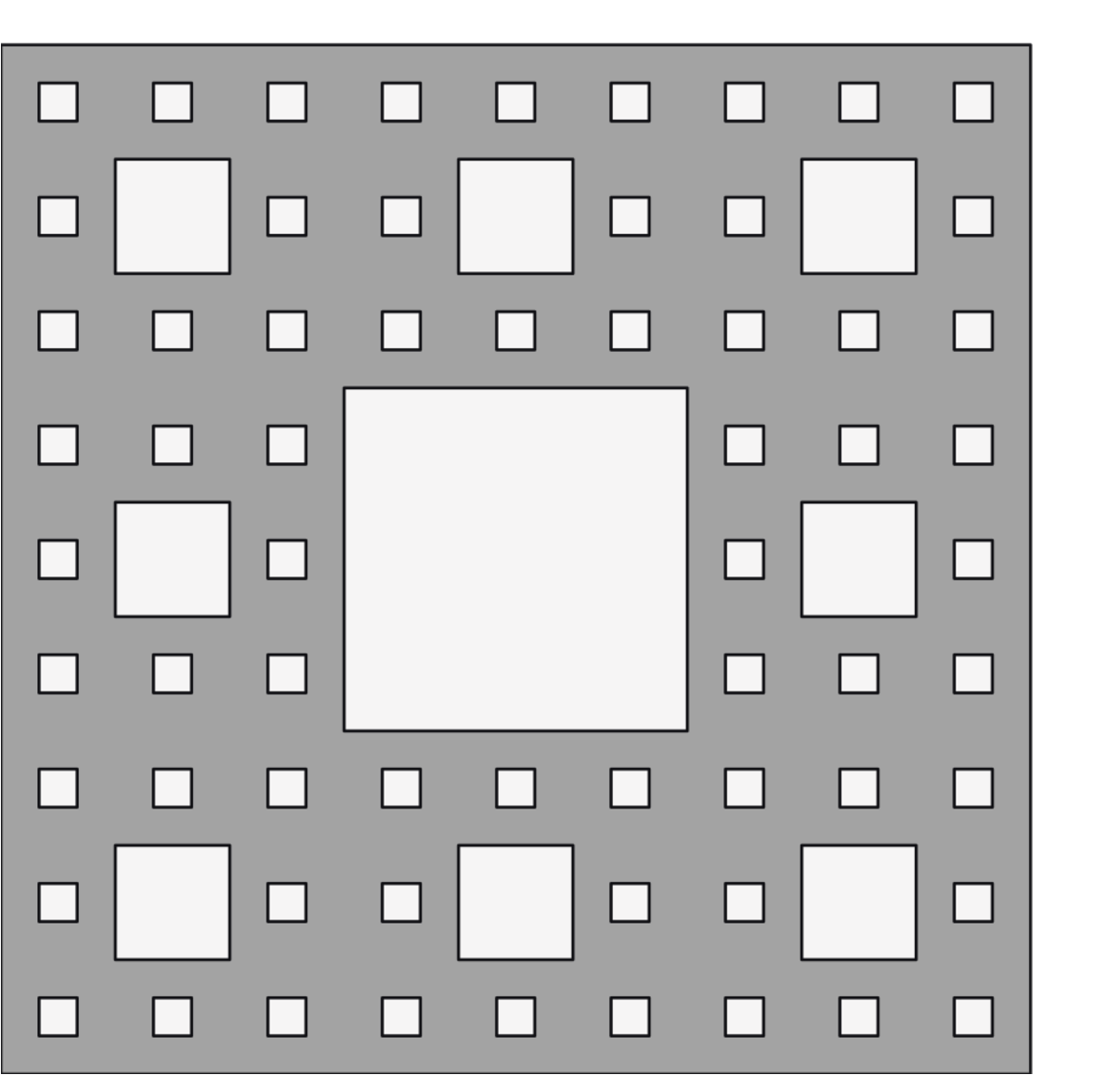}
	\end{center}
	\caption{\small{First three steps in the construction of the  Sierpi\'nski carpet $\mathbb{S}_{1/3}$.}}\label{standardcarpet}
\end{figure}

The study of the quasisymmetric geometry of metric carpets has received considerable attention in recent decades.  See, for in stance, \cite{Bon06, Bon11, BKM09, BM13, Hak22,HL23, Mer10, MW11, MTW13} for the study of carpets in metric geometry, \cite{BLM16} for the study of carpets in dynamics and \cite{Hai15, Kle06, KK00} for the study of carpets in geometric group theory. One underlying motivation is the Kapovich-Kleiner conjecture \cite{KK00} in geometric group theory, a formulation of which says that each metric carpet that is the boundary at infinity of a Gromov hyperbolic group is a quasiround carpet. Bonk provided mild sufficient conditions for a metric carpet $X \sub \C$ to be quasiround, showing that if the boundaries of the complementary components of $X$ are uniformly relatively separated uniform quasicircles, then the carpet is quasiround \cite[Theorem $1.1$]{Bon11}.

A reasonable first \emph{random} candidate to consider in this regard is the conformal loop ensemble $\CLE_\kappa,  \kappa \in (8/3,8)$, introduced by Sheffield and Werner \cite{SW12, She09}. Intuitively, $\CLE_\kappa$ is a collection of random loops in a simply connected domain $\Omega$, where each loop locally looks like a $\SLE_\kappa$ curve. We define the \emph{$\CLE_\kappa$ space}, $\kappa\in(8/3,8)$, as the set of points in $\Omega$ which are not surrounded by any loop. In this paper, we focus on the case $\kappa\in(\frac{8}{3},4]$ where every loop in $\CLE_\kappa$ is simple and does not intersect with each other.

\begin{Theorem}\label{topcarpet}
    For $\kappa \in (\frac{8}{3}, 4]$, the $\CLE_\kappa$ space is a.s. a metric carpet.
\end{Theorem}

\noindent Thus we are justified in calling the $\CLE_\kappa$ space the \emph{$\CLE_\kappa$ carpet} for $\kappa \in (\frac{8}{3}, 4]$. However, as in the case of arcs, we have the following negative result, which is consistent with Theorem \ref{SLE}.

\begin{Theorem}\label{QScarpet}
    For $\kappa \in (\frac{8}{3}, 4]$, the $\CLE_\kappa$ space is a.s. not a quasiround carpet.
\end{Theorem}

\begin{Remark}
	We essentially prove a result stronger than Theorem \ref{QScarpet}, namely, that a.s. no boundary circle of the $\CLE$ carpet is a quasicircle. See Theorem \ref{Theorem: nonquasiarc}.
\end{Remark}

\subsection{Quasispheres}\label{Sec:IntroSphere}

For the sake of completeness we also briefly overview the case of topological spheres, although our paper makes no new contributions here.  Recall a metric space is a \emph{quasisphere} if it is quasisymmetric to $\mathbb{S}^2$. Similar to the case of metric carpets, quasispheres have elicited substantial attention in recent decades.  See, e.g., \cite{BK02,  Mey10, MW11, Wil08} for the study of spheres in metric geometry and \cite{Bon06, BK05, Can94, Kle06} for the study of spheres in geometric group theory. Analog to the carpet case, there is also a motivation from geometric group theory, where Cannon's famous conjecture states that each sphere that is the boundary at infinity of a Gromov hyperbolic group is a quasisphere \cite{Can94}. Bonk and Kleiner \cite[Theorem~$1.1$]{BK02} proved that an Ahlfors $2$-regular metric sphere is a quasisphere if and only if it is linearly locally contractible. Given this result, they showed in \cite[Theorem~$1.1$]{BK05} that Cannon conjecture holds if the boundary attains its minimal Ahlfors regular conformal dimension .

Are there natural random spheres that are, in fact, quasispheres?  A natural candidate is Liouville quantum gravity $\LQG_\gamma$, $\gamma \in (0,2)$, which is homeomorphic to $\mathbb{S}^2$. LQG was introduced by \cite{DS11} and has been widely studied in the last decade \cite{DDDF20, GM21}. As in our results above, a negative answer for the quasisymmetric uniformization question for $\LQG_\gamma$ is given by \cite{Tor21} and \cite{Hug24}, who show that a.s. $\LQG_\gamma$ for $\gamma \in (0,2)$ is not a doubling metric space. Since the doubling property is a quasisymmetric invariant \cite[Theorem~$10.18$]{Hei01} and $\mathbb{S}^2$ is doubling, $\LQG_\gamma$ is a.s. not a quasisphere.

 \subsection{Discussion} The consistent theme in all our examples is that many natural random fractals fall outside the realm of quasisymmetric uniformization to simple canonical spaces. This may lead to an intuition that the quasi-geometrical world is disjoint with the stochastic world. However, we believe that there are some deeper idea behind them and these two worlds can be interconnected through an exploration of the renowned \emph{Sullivan dictionary}. The Sullivan dictionary establishes a bridge for translating objects, theorems, and conjectures between Kleinian groups and dynamics of holomorphic maps. Initially observed by Fatou during the early days of complex dynamics, it was rediscovered and formalized by Sullivan in the 1980s \cite{Sul83, Sul85a, Sul85b}. It has been extended to a wider range of analogies between geometry and dynamics. In some cases, very similar proofs can be given of related results in these two fields. More commonly, the dictionary suggests loose analogies which motivate research in each area. Such a machinery is strongly effective, leading to the resolution of several notable conjectures. The following table provides a quick glance of the dictionary:
 \begin{small}
 \[
   	\begin{array}{|c|c|}
			\hline
			\textnormal{\textbf{Geometry}} & \textnormal{\textbf{Dynamics}}
			\\ \hline
			\textnormal{A finitely generated Kleinian group} & \textnormal{A rational map}
			\\ \hline
			\textnormal{convex-cocompact Kleinian groups} & \textnormal{rational maps exhibiting expansion properties}
			\\ \hline
			\textnormal{The domain of discontinuity } & \textnormal{Fatou set}
			\\ \hline
			\textnormal{Limit set} & \textnormal{Julia set}
			\\ \hline
			\textnormal{Fixed points} & \textnormal{Preperiodic points}
			\\ \hline
			\textnormal{Patterson-Sullivan Measures} & \textnormal{Sullivan's conformal measures}
			\\ \hline
			\textnormal{Ahlfors Finiteness Theorem} & \textnormal{Sullivan's no wandering domains theorem}
			\\ \hline
			\textnormal{Bers Area Theorem} & \textnormal{Shishikura's bound on periodic orbits}
			\\ \hline
			\textnormal{Cannon Conjecture} & \begin{split}
				\textnormal{Thurston’s Characterization theorem for}
				\\
				\textnormal{postcritically finite rational maps}
			\end{split}
			\\ \hline
			\textnormal{Mostow rigidity theorem} &
			\begin{split}
				\textnormal{Thurston’s Uniquness theorem for}
				\\
				\textnormal{postcritically finite rational maps}
			\end{split}
			\\ \hline
			\textnormal{The ending lamination theorem} & \textnormal{The Mandelbrot set locally connected conjecture}
			\\ \hline
		\end{array}
 \]
 \end{small}
 
In recent years, there has been growing interest in extending this dictionary to include a third column--\textbf{Statistic Geometry}. Intuitively, the metric spaces generated by probabilistic phenomena often exhibit a random ``self-similar'' property, which shares the common philosophy with the boundaries of hyperbolic groups or Julia sets. Additionally, scaling limits and the Markov property also appears in an analogous sense within both geometric group theory and complex dynamics. 
 
Our paper explores the probabilistic aspect of this framework by demonstrating that probabilistic objects may not align well with ``regular'' metric geometry. Thus to further develop this analogy, it is essential to first identify the canonical uniformization space and the canonical isomorphism within the context of random structures.

\subsection{Organization}
This paper is organized as follows. Section \ref{Preliminaries} gives basic definitions and notations and then briefly introduces Brownian motion, $\SLE_\kappa$, and $\CLE_\kappa$. Section \ref{Section: QGRA} is devoted to proving Theorems \ref{BM} and \ref{SLE}, and we prove Theorems \ref{topcarpet} and \ref{QScarpet} in Section \ref{Sec:Carpets}. We record some observations and formulate
several open problems in Section \ref{sec:FQ}.

\subsection*{Acknowledgments}
The authors would like to thank Ilia Binder for his valuable insights in Section \ref{Section: QGRA}. Hrant Hakobyan and Ville Suomala provided helpful suggestions on fractal percolation. The authors also greatly appreciate the many suggestions from Xinyi Li.
\section{Preliminaries}\label{Preliminaries}

As usual, $\C$ is the complex plane equipped with Euclidean metric, $\H$ is the open upper half plane of $\C$, $\R$ is the real line equipped with Euclidean metric, $\Q$ is the rational numbers and $\N$ is the positive integers. 

For a topological space $X$ with  $A \subseteq X$, we write $\overline{A}$, $A^{\mathrm{o}}$ and $\partial A$ for the topological closure, interior and boundary of $A$, respectively. 

Let $E$ be a subset of a metric space $(X, d)$. The diameter $\diam(E)$ of $E$ has the usual definition, as follows:
\begin{align*}
	\operatorname{diam}(E) &\=\sup\{d(x,y) : x,y \in E\}.
\end{align*}

Quasisymmetric mapping introduced in Section \ref{Section: Introduction} is the main object studied in this paper. Additionally, we also introduce quasiconformal mappings, which is closely related to quasisymmetries.

Let $f: X \to Y$ be a homeomorphism between two metric spaces $(X, d_X)$ and $(Y, d_Y)$. For a point $x \in X$ and $r > 0$, we define the \emph{linear dilatation of $f$ at $x$} as
\begin{align}
    H_f(x)=\limsup_{r \to 0}\frac{\sup_{y}  \{ d_Y(f(x),f(y)) \, | \, d_X(x,y)\leq r\}}{\inf_{y}   \{ d_Y(f(x),f(y)) \, | \, d_X(x,y)\geq r\}}.
\end{align}

We say that a homeomorphism $f: X\to Y$ is \emph{$H$-quasiconformal} if
\begin{align*}
  \sup_{x \in X}H_f(x) \leq H
\end{align*}
for some $1\leq H <\infty$. A map is quasiconformal if it is $H$-quasiconformal for some $H$. Remark that any conformal map is a $1$-quasiconformal mapping.

Given a quasiconformal map, we may ask if it is a quasisymmetry. The following statement shows an assertion of it. The proof directly follows \cite[Theorem~$11.14$]{Hei01} and \cite[Theorem~$2.23$]{TV80}.

\begin{Theorem}\label{compactQS}
    Let $f: U \to V$ be a $H$-quasiconformal map between two domains in $\R^n$ and $A \sub U$ be a compact subset. Then $f|_A$ is an $\eta$-quasisymmetry, where $\eta$ depends on $H, n$ and $A$.
\end{Theorem}

\subsection{Brownian motion and Schramm-Loewner Evolution}

As mentioned in the introduction, \emph{one-dimensional Brownian motion} is a continuous stochastic process with stationary increments that are independent and Gaussian. We denote one-dimensional Brownian motion defined on the probability space $\left( \Omega, \mathcal{F}, \mathbb{P}\right)$ by $B(t)$. See \cite[Chapters $1$ and $2$]{MP10} for precise definitions and basic properties.  Collecting $n$ independent one-dimensional Brownian motions $\B(t) = (B_1(t), \ldots, B_n(t))$ yields a \emph{$n$-dimensional Brownian motion}. A Brownian motion is called a \emph{standard Brownian motion} if it starts from the origin point.

Brownian motion possesses the following \emph{Markov property} \cite[Theorem~$2.3$]{MP10}. Let $\B(t)$ be a $n$-dimensional Brownian motion. Assume that $s > 0$, then the process
\begin{equation}\label{Markov}
    \{\B(t+s) - \B(s): t > 0\}
\end{equation}
is a Brownian motion started in the origin and it is independent of $\B(t)$.

In this paper, we consider several natural variants of $\SLE_\kappa$, namely the \emph{chordal} SLE, the \emph{radial} SLE, the \emph{whole-plane} SLE and the \emph{$\SLE_\kappa(\rho)$} processes. These variants naturally appear when considering different settings of two-dimensional critical statistical physics models. Excellent references on these $\SLE_\kappa$'s include \cite{Law05,Kem17} and \cite[Appendix A and B]{BP23}. For whole-plane $\SLE_\kappa$, see \cite{MS17}.

Chordal \emph{Schramm-Loewner Evolution} $\SLE_\kappa$, $\kappa \in [0,\infty)$, is a one-parameter family of conformally-invariant random planar curves connecting two prime ends in a simply connected domain $\Omega \subsetneq \C$. It was introduced by Schramm in \cite{Sch00} and has seen significant development over the past decade. It is a random growth process defined by Loewner equation with a driving parameter given by a standard one-dimensional  Brownian motion running with speed $\kappa$. Specifically, the \emph{chordal Schramm-Loewner evolution} with parameter $\kappa \geq 0$ in $\H$ is the solution to the following equation.
\begin{equation} \label{chordalSLE}
	\partial_t g_t(z) = \frac{2}{g_t(z) - \xi_t} , \
	g_0(z) = z , 
\end{equation}
where $\xi_t = \sqrt{\kappa}B_t$ and $B_t$ is the standard one-dimensional Brownian motion. Intuitively, SLE can be viewed as the Brownian motion on the space of conformal maps. The  \emph{$\SLE_\kappa$ curve} $\gamma$ is defined as follows:
\[
\gamma(t) = \lim_{z \to 0}g_t^{-1}(z+ \sqrt{\kappa} B_t)
\]
where $z$ tends to $0$ within $\H$. If the limit does not exist, let $\gamma(t)$ denote the set of all limit points.Notably, $\SLE_\kappa$ has the following two properties:
\begin{enumerate}

    \item \textbf{Conformal Invariance}. If $\gamma$ is a $\SLE_\kappa$ curve and $\phi$ is a conformal map of the upper half plane that fix $0$ and $\infty$, then $\phi(\gamma)$ is also a $\SLE_\kappa$ curve.
    
    \item \textbf{Domain Markov Property}. If $\gamma$ is a $\SLE_\kappa$, $T$ is a finite stopping time and $\phi$ is a conformal map from $\mathbb{H} \setminus \gamma(0, T)$ to  $\mathbb{H}$ that maps $\gamma(T)$ to $0$ and fix $\infty$ . Then $\phi(\gamma)$ is a $\SLE_\kappa$ curve.
\end{enumerate}
For general simply connected domain $D \sub \C$ and $a,b\in\partial D$, we define the chordal $\SLE_\kappa$ from $a$ to $b$ to be the law of the chordal $\SLE_\kappa$ on $\H$ from $0$ to $\infty$ under a under the conformal map map $\phi:\H\to D$ with $0,\infty$ mapped to $p,q$. Note that this definition does not rely on the choices of $\phi$ by the above conformal invariance of the chordal $\SLE_\kappa$ on $\H$.

Compared with chordal $\SLE_\kappa$, the radial $\SLE_\kappa$ in a domain describe processes that start from the boundary but are targeted at an interior point of the domain. Fix the unit disk $\D$ and choose starting and targeting points as $1$ and $0$, respectively. The \emph{radial Schramm-Loewner evolution} with parameter $\kappa \geq 0$ in $\D$ is the solution of the following radial Loewner equation
\begin{equation*}
	\partial_tg_t(z)=g_t(z)\frac{\xi_t+g_t(z)}{\xi_t-g_t(z)},\quad g_0(z)=z
\end{equation*}
where $\xi_t=e^{i\sqrt{\kappa}B_t}$ and $B_t$ is the standard one-dimensional Brownian motion. The \emph{radial $\SLE_\kappa$ curve} is defined analogously to the chordal case.
In general, for a simply connected domain $D$ and $p \in \partial D$, $q\in D$, the radial $\SLE_\kappa$ on $D$ from $p$ to $q$ is defined as the image of the radial $\SLE_\kappa$ on $\D$ from $1$ to $0$ under the conformal map $\phi$ from $\D$ to $D$ such that $1,0$ are mapped to $p,q$.

The \emph{Whole-plane} $\SLE$ is another variant of $\SLE$ which can be thought of as a bi-infinite time version of radial $\SLE$. It describes a random growth process $K_t$ where, for each $t \in \R$, $K_t \subseteq \C$ is compact with $\C_t = \C \setminus K_t$ simply connected (viewed as a subset of the Riemann sphere).  For each $t$, we let $g_t \colon \C_t \to \C \setminus \D$ be the unique conformal transformation with $g_t(\infty) = \infty$ and $g_t'(\infty) > 0$.  Specifically, the \emph{whole-plane} $\SLE_\kappa$ from $0$ to $\infty$ is the solution of
\begin{equation}
	\label{eqn::whole_plane_loewner}
	\partial_t g_t = g_t(z) \frac{\xi_t + g_t(z)}{\xi_t - g_t(z)}.
\end{equation}
where $\xi_t=e^{i\sqrt{\kappa}B_t}$ and $B_t$ is the standard one-dimensional Brownian motion. 

For any $p, q \in \C$, we also define whole-plane $\SLE_\kappa$ from $p$ to $q$ to be the image of whole-plane $\SLE_\kappa$ from $0$ to $\infty$ under a M\"obius transform $\phi$ mapping $0,\infty$ to $p,q$. Note that this is well-defined since the law of the whole-plane $\SLE_\kappa$ from $0$ to $\infty$ is scaling invariant.

We also present an alternative construction of the whole-plane SLE from $0$ to $\infty$. Let $\mu_{a}$ be the distribution of the radial $\SLE_\kappa$ in $a \D$ connecting $a$ and $0$. We define $\nu_{a }$ as the distribution of random curves as the image of the samples of $\mu_{a }$ under the map $1/az$. Then the distribution $\nu$ of the whole-plane $\SLE_\kappa$ is the weak limit of $\nu_{a}$ as $a \to 0$. For more details, we refer the reader to \cite{MS17, Zha15}.

For each SLE variant mentioned above, we can construct its corresponding \emph{$\SLE_\kappa(\rho)$ process}, $\rho \in \R$. Specifically, the chordal $\SLE_\kappa(\rho)$ process on $\H$ is defined as the random process generated by the associated Loewner equation above with a driving function $W_t$ satisfying
\begin{equation}\label{Equation: cSLErho}
d W_t = \sqrt{\kappa} d B_t + \frac{\rho}{W_t - V_t}  \qquad d V_t = \frac{2}{V_t - W_t} dt.
\end{equation}
Refer to \cite[Section~$8$]{LSW03} for background and motivation. For radial SLE on $\D$ and whole-plane $\SLE_\kappa(\rho)$ from $0$ to  $\infty$, the driving function is defined as
\begin{equation}\label{Equation: rSLErho}
d W_t = \sqrt{\kappa} d B_t + \frac{\rho}{2} \cot\left( \frac{W_t - V_t}{2} \right) dt\qquad d V_t = -\cot\left( \frac{W_t - V_t}{2} \right) dt.
\end{equation}
We refer the reader to see \cite[Section~$2$]{Wu15} and \cite[Section~$2$]{MS17} for detailed constructions. The corresponding $\SLE_\kappa(\rho)$ from $p$ to $q$ for any $p, q \in \C$ is defined analogously as above.

The significance of SLE lies in its ability to describe the scaling limits of many interfaces at criticality in two-dimensional statistical physics. For example, under appropriate boundary conditions and lattice configurations, the interfaces in critical Ising and critical percolation converge to $\SLE_3$ and $\SLE_6$, respectively. The parameter $\kappa$ describes how rough the fractal curve $\gamma_\kappa$ is. When $\kappa =0$, $\gamma_0$ is the deterministic hyperbolic geodesic between the two prime ends;  it is simple when $\kappa \leq 4$ and non-simple otherwise, and is space-filling when $\kappa \geq 8$.

\subsection{The conformal loop ensemble}\label{Sec:CLE}

There is a loop version of $\rm SLE_\kappa$, called the \emph{Conformal Loop Ensemble} CLE$_\kappa$, $\kappa\in(8/3,8)$, which is the natural candidate for the full scaling limit of the collection of all interfaces at criticality of many two-dimensional statistical physics models. For instance, critical Ising model corresponds to $\CLE_3$, and critical percolation to $\CLE_6$. As $\kappa$ varies there are two regimes. The \emph{dilute case}, defined in \cite{SW12} is when each loop in $\rm CLE_\kappa$ is simple and does not intersect other loops which occurs for $\kappa\in(8/3,4]$. When $\kappa\in(4,8)$ one has the {\it dense case}, which is defined in \cite{She09} and the $\CLE_\kappa$ loops which touch both themselves and others. Since we can only naturally obtain a topological carpet in the dilute case, in this paper we focus exclusively on $\kappa\in(8/3,4]$.

Given a simply connected domain $D \subsetneq \C$, a \emph{loop configuration} is a countable collection of disjoint simple loops where, for each $\varepsilon>0$, the number of loops with diameter larger than $\varepsilon$ is finite. Sheffield and Werner \cite[Theorem 1.1]{SW12} showed that, for $\kappa \in (8/3, 4]$, $\rm CLE_\kappa$ is the unique one-parameter family of probability distributions on loop configurations of non-nested loops in $D$ possessing the following two properties:

\begin{enumerate}
\item \textbf{Domain Markov Property}. For a simply connected subdomain $U\sub D \sub \mathbb{C}$, the conditional law of the loops in $U$, given those loops not contained in $U$, is $\rm CLE_\kappa$ in the (connected components of the) remaining domain.

\item \textbf{Conformal Invariance}. If $\phi:D \to D'$ is a conformal map and $\Gamma$ is a sample of $\mathrm{CLE}_\kappa$ in $D$, then the law of $\phi(\Gamma)$ is $\mathrm{CLE}_\kappa$ in $D'$.
\end{enumerate}
We use the terminology \emph{$\CLE_\kappa$ configuration} to emphasize the collection of random loops in $\CLE_\kappa$.

For $\kappa\in(\frac{8}{3},4]$, $\CLE_\kappa$ may be constructed via \emph{Brownian loop soup}, as in \cite[Section 8]{SW12}, and as we explain in the following. First, we define the \emph{Brownian loop measure} $\mu_D$ as follows: For any measurable subset $A$ of the space of loops on $D$,
\begin{equation*}
\mu_D(A) :=\int_D\int_0^\infty \frac{1}{2\pi t^2}\mathbb{P}^t_z(A) \,dt\, dz,
\end{equation*}
where $\mathbb{P}^t_z(\cdot)$ denotes the law of the 2-dimensional Brownian bridge in $D$ from $z$ to $z$ with time length $t$, i.e., ${\B}_s-\frac{s}{t}{\B}_t, 0 \leq s\leq t$, where $\bf B$ is a 2-dimensional Brownian motion. The Brownian loop measure is a $\sigma$-finite measure on the space of loops on $D$. The Brownian loop soup $\mathcal{L}$ on $D$ with intensity $c$ is then the Poisson point process with intensity measure $c\mu_D$. Roughly speaking, it is a measure defined on the spaces of a countable family of loops on $D$. See \cite{LW04} for additional background.

We say two Brownian loops $l$ and $l'$ are in the same cluster if there is a sequence of loops $l_0=l,l_1,....,l_n=l'$ such that $l_i\cap l_{i+1}\neq\emptyset$, $0\le i\le n-1$. According to \cite[Proposition 10.2]{SW12}, the collection of boundaries of the outermost clusters of a Brownian loop soup with intensity $c\mu_D, c \in (0,1]$, has the law of $\CLE_\kappa$ on $D$ for $c=\frac{(3\kappa-8)(6-\kappa)}{2\kappa}$.

\section{Quasisymmetric geometry of random arcs}\label{Section: QGRA}

In this section we prove Theorems \ref{BM} and \ref{SLE}, showing that the trace of Brownian motion, the graph of Brownian motion and various $\SLE_\kappa$ variants are a.s. not quasiarcs.

We begin with Brownian motion. Recall from the introduction the Tukia-V\"ais\"al\"a result that quasiarcs are precisely the metric spaces which are doubling and of bounding turning. Since the Brownian motions and Brownian graphs reside in Euclidean spaces, which are doubling, to show they are a.s. not a quasiarc we need to show they are a.s. not of bounded turning.  The argument begins with the following lemma.

\begin{Lemma}\label{Lemma: maxpoint}
    Let $\B(t)$ be a $n$-dimensional Brownian motion.  Fix $ a > 0$. Then a.s. for sufficiently large $j$, there exists an interval $[\frac{i}{2^j}, \frac{i+1}{2^j}]$ for some $1\le i\le 2^j-1$ satisfying the following conditions:
    \begin{enumerate}
        \item $\left|{\bf B}(\frac{i+1}{2^j}) - {\bf B}(\frac{i}{2^j})\right| \leq \frac{2}{\sqrt{2}^j}$
        
        \item $\max_{\frac{i}{2^j} \leq t \leq \frac{i+1}{2^j}}\left|{\bf B}(t)-{\bf B}(\frac{i}{2^j})\right| \geq \frac{a}{\sqrt{2}^j}$
    \end{enumerate}
\end{Lemma}

\begin{proof}
For each $i, j \geq 1$ and $0\le i\le 2^j-1$, let $E_{i,j}$ be the event that
\[
\left|{\bf B}(\frac{i+1}{2^j})-{\bf B}(\frac{i}{2^j})\right|\le \frac{2}{\sqrt{2^j}}
\]
and $F_{i,j}$ be the event that
\[
\max_{\frac{i}{2^j} \leq t \leq \frac{i+1}{2^j}}\left|{\bf B}(t)-{\bf B}(\frac{i}{2^j})\right| \geq \frac{a}{\sqrt{2}^j}.
\] 

We claim that  there exists a constant $C \in (0,1)$ depending on $a$ such that for each $i, j \in \N$ and $0\le i\le 2^j-1$, 
\[
\P(E_{i,j} \cap F_{i,j}) = 1-C.
\]

It follows the Markov property \ref{Markov} and the scaling property of Brownian motion \cite[Lemma~$1.7$]{MP10} that for each $i,j \in \N$ and $0\le i\le 2^j-1$,
\[
\P(E_{i,j} \cap F_{i,j}) = \P\left( |\B(1)| \leq 2,\ \max_{0\le t\le1}|{\bf B}(t)|\ge a \right ).
\]
Thus it is sufficient to prove that $\P(E_{i,j} \cap F_{i,j}) > 0$.

\begin{align*}
		\P(E_{i,j} \cap F_{i,j}) & = \P\left( |\B(1)| \leq 2,\ \max_{0\le t\le1}|{\bf B}(t)|\ge a \right ) 
		\\
		& \geq \P\left( B\left(\frac{1}{2}\right) \geq a  \right) \P\left(   |B(1)|  \leq 2   \ \Big|  \ B\left(\frac{1}{2}\right) \geq a  \right) 
		\\
		& > C_1
\end{align*}
for some $C_1 > 0$. This finishes the proof of the claim.

Since $\P\left( \cap_{i=0}^{2^j-1} (E_{i,j} \cap F_{i,j})^c \right) = C^{2^j}$, we have
\[
\sum_{j=1}^\infty \P\left( \cap_{i=0}^{2^j-1} (E_{i,j} \cap F_{i,j})^c \right)  < \infty.
\]

We conclude that for sufficiently large $j$, $\cup_{i=0}^{2^j-1} (E_{i,j} \cap F_{i,j})$ occurs by Borel-Cantelli lemma. Namely, a.s. for sufficiently large $j$, there is some $0\le i\le 2^{j-1}$ such that $E_{i,j} \cap F_{i,j}$ occurs. This finishes the proof.
\end{proof}

We are now ready to prove Theorem \ref{BM}.

\begin{proof}[Proof of Theorem~\ref{BM}]
    Fix $n \in \N$. Let $\B(t)$ be a $n$-dimensional Brownian motion. Fix $a > 0$ and $\frac{i}{2^j}, \frac{i+1}{2^j}$ be the corresponding time in Lemma \ref{Lemma: maxpoint}. Then
    \[
    \frac{\diam\left(\B\left( \left[ \frac{i}{2^j}, \frac{i+1}{2^j} \right] \right)  \right)}{ \left|\B(\frac{i+1}{2^j}) - \B(\frac{i}{2^j})\right| } \geq \frac{a}{\sqrt{2}^j} \Big/ \frac{2}{\sqrt{2}^j} = \frac{a}{2}.
    \]
    Letting $a \to \infty$, we have a sequence of pairs of points on the trace of $\B$ such that distance between them is much smaller than the diameter of the arc between them. Thus $\B([0,1])$ is a.s. not bounded turning and thus a.s. not a quasiarc by \cite[Theorem~$4.9$]{TV80}.
    
    It then follows the Markov property \ref{Markov} that any a.s. arc on the trace of $\B$ is not a quasiarc.
    
    We now turn to the graph of $\B$. Similarly, fix $a > 0$ and let the corresponding time $\frac{i}{2^j}, \frac{i+1}{2^j}$ in Lemma \ref{Lemma: maxpoint}. Let $\gamma(t) = (t, \B(t)) \in \R^n$. Then for sufficiently large $j$,
    \[
    \frac{\diam\left(\gamma\left( \left[ \frac{i}{2^j}, \frac{i+1}{2^j} \right] \right)  \right)}{ |\gamma(\frac{i+1}{2^j}) - \gamma(\frac{i}{2^j})| } \geq \frac{a}{\sqrt{2}^j} \Big/ \left( \frac{2}{\sqrt{2}^j} + \frac{1}{2^j} \right) \geq \frac{a}{3}.
    \]
    
    Letting $a \to \infty$, we have a sequence of pairs of points on $\gamma$ such that distance between them is much smaller than the diameter of the arc between them. Then $\gamma([0,1])$ is a.s. not bounded turning and thus a.s. not a quasiarc \cite[Theorem~$4.9$]{TV80}.
    
    It then follows the Markov property \ref{Markov} that a.s. any arc on the graph of $\B$ is not a quasiarc.
\end{proof}

Turning to SLE, the proof of Theorem~\ref{SLE} relies on a different method. Recall that $\SLE$ is a random curve generated by the Loewner equation with driving function a scaled Brownian motion. In fact, the Holder continuity exponent of the driving function will determine the quasisymmetric geometry of the curve. See \cite[Theorem~$1.1$ and $1.2$]{MR05}.

\begin{proof}[Proof of Theorem~\ref{SLE}]
	
	We first prove that Theorem \ref{SLE} holds for radial SLE on $\D$. Let $\gamma \: [0, \infty) \to \D$ be a parametrization of radial $\SLE_\kappa$ on $\D$. 	Recall that $\gamma$ is generated from the radial Loewner equation with the driving function of a scaled Brownian motion. Fix $t > 0$ and let $\gamma_t = \gamma([0,t])$. Since Brownian motion is a.s. not $1/2$-H\"older by L\'evy’s modulus of continuity \cite[Theorem~$1.14$]{MP10}, it follows \cite[Theorem~$1.2$ and Lemma $2.3$]{MR05} that a.s. $\gamma_t$ is not the image of a segment under a quasiconformal mapping $g: \C \to \C$. Ahlfors pointed out that a curve is the image of a quasiconformal mapping from $\C$ to $\C$ if and only if it is bounded turning, see \cite{Ahl63} or \cite{Ahl06}. Thus $\gamma_t$ is not bounded turning and not a quasiarc by \cite[Theorem~$4.9$]{TV80}.
	
    Following the same procedure and applying \cite[Theorem 1.1]{MR05}, we have the following result: Let $\gamma : [0, \infty) \to \H$ be a parametrization of chordal $\SLE_\kappa$ on $\H$, then $\gamma([0,t])$ is not a quasiarc a.s. for any $t > 0$.
	
    We now focus on the whole-plane SLE.  Note that any $\gamma_t$ is not a quasiarc a.s. where $\gamma_t = \gamma([0,t])$, as discussed above, is subcurve of radial $\SLE_\kappa$ on $\D$. Recall that $\nu_a$ is the image of the radial $\SLE_\kappa$ on $a \D$ from $a$ to $0$ under the map $f(z)=1/z$. Since $f(z) = 1/z$ is a quasisymmetry on any neighborhood of $z=1$ by Theorem \ref{compactQS}, we have
	\[
	\nu_a(\Omega_a) = 1
	\]
    where $\Omega_a$ is the $\nu_a$-event that the sample curve is not a quasiray.
	
    Note that the law $\nu$ of whole-plane $\SLE_\kappa$ from $0$ to $\infty$ is obtained as the weak limit of $\nu_a$ as $a\to0$. Let $\Omega_\infty = \cap_{n=1}^\infty \Omega_{1/n}$. Then $\Omega_\infty$ is a subset of the $\nu$-event that the sample curve is not a quasiray. Moreover, $\nu(\Omega_\infty) = 1$. This implies that the whole-plane SLE is not a quasiray a.s.. 
    
    Following the scaling invariant property of $\SLE$, we have the following general result: Let $\gamma: \R \to \C$ be a parametrization of whole-plane $\SLE_\kappa$ from $0$ to $\infty$, then $\gamma([a, b])$ is a.s. not a quasiarc for any $[a,b] \sub \R$.
	
    Recall that the driving function of a chordal $\SLE_\kappa(\rho)$ is define in equation \eqref{Equation: cSLErho} and the driving function of radial $\SLE_\kappa(\rho)$ and whole-plane $\SLE_\kappa(\rho)$ is define in equation \eqref{Equation: rSLErho}. Thus to prove that the chordal, radial and whole-plane $\SLE_\kappa(\rho)$ is not a.s. a quasiray, quasiarc and quasiray, respectively, it is sufficient to show that their driving functions are not $1/2$-H\"older. This is not hard to see since for all the three cases, the driving function $W_t$ are mutually absolutely continuous to $B_t$ in every small neighbourhood when $W_t \neq V_t$. This finishes the proof.
\end{proof}

\begin{Remark}\label{SLEgeneral}
	Notice that any continuous curve of whole-plane $\SLE_\kappa$ or whole-plane $\SLE_\kappa(\rho)$ from $0$ to $\infty$ is not a quasiarc a.s. as established in the proof of Theorem \ref{SLE}. Consequently, it follows from Theorem \ref{compactQS} that any continuous curve of whole-plane $\SLE_\kappa$ or whole-plane $\SLE_\kappa(\rho)$ from $p$ to $q$ is a.s. not a quasiarc for any $p ,q \in \C$.
\end{Remark}

\section{Quasisymmetric geometry of the CLE carpet}\label{Sec:Carpets}

To prove Theorem \ref{topcarpet}, we need the following result of Whyburn, see \cite{Why58} or \cite[Theorem~$5.1$]{HL23}, which characterizes topological carpets in $\C$.

\begin{Theorem}[Whyburn]\label{metriccarpet}
    Suppose $D$ is a simply connected domain in $\C$. Let $D_n \sub D, \ n \geq 0$ be a sequence of topological open disks satisfying the following conditions:
	\begin{enumerate}
		\item $\overline{D_i} \cap \overline{D_j} = \emptyset, \ \textrm{for} \ i \neq j$;
		
		\item $\diam(D_n) \to 0, \ \textrm{as} \ n \to \infty$;
		
		\item $\overline{\bigcup_n D_n} = \overline{D}$.
	\end{enumerate}
    Then the compact set $\overline{D} \backslash \bigcup_n D_n$ is homeomorphic to the standard Sierpi\'nski carpet $\mathbb{S}_{1/3}$.
\end{Theorem}

\begin{proof}[Proof of Theorem \ref{topcarpet}]

Let $\Gamma = \{\gamma_n\}_{n=1}^\infty$ be a $\CLE_\kappa$ configuration on a simply connected domain $D$ for a fixed $\kappa \in (\frac{8}{3}, 4]$, and we assume that the loops $\gamma_n$ are ordered in a way such that $\diam(\gamma_n) \geq \diam(\gamma_{n+1})$ for all $n \in \N$. Since $\CLE_\kappa$ configuration is a random loop configuration, we have the following conditions a.s.:
\begin{enumerate}
    \item $\{\gamma_n\}$ is a collection of disjoint non-nested simple loops;

    \item $\diam(\gamma_n)\to 0$ as $n\to\infty$.
\end{enumerate}
Setting $D_n$ as the open topological disk in $D$ such that $\partial D_n = \gamma_n$, then conditions $(1)$ and $(2)$ in Theorem \ref{metriccarpet} are naturally satisfied.

In the next step, we show that a.s. $\overline{\bigcup_n D_n} = \overline{D}$. Let $\mathcal{L}$ be a Brownian loop soup with intensity $c \mu_D$ in $D$, and let $L$ be the collection of all loops of $\mathcal{L}$. Fix a $z \in D$ and a $\varepsilon>0$, we claim that a.s. there exists a loop of $L$ in $B(z, \varepsilon)$ where $B(z, \varepsilon) := \{w \in D: |w-z|<\varepsilon\}$. In fact, it follows \cite[Proposition 14]{LW04} that the sum of all loops of $L$ that are in $B(z, \varepsilon)$ is infinite, thus $\card(\{l : l \in L, l \sub B(z, \varepsilon)\})$ is infinite.

Recall that  $\CLE_\kappa$ can be obtained as the outer boundaries of Brownian loop clusters \cite[Proposition~$10.2$]{SW12},  this implies that a.s. 
\[
B(z, \varepsilon) \cap \left(\bigcup_{n=1}^\infty D_n\right)\neq\emptyset.
\]
Then we have a.s. for all $\varepsilon\in \mathbb{Q} \cap (0, 1)$,
\[
B(z, \varepsilon) \cap \left(\bigcup_{n=1}^\infty D_n\right)\neq\emptyset.
\]
Thus  $z \in \overline{\bigcup_n D_n}$. Repeating this for all $z \in D$ with rational coordinates, we have a.s. $\overline{\bigcup_n D_n} = \overline{D}$.

According to Theorem \ref{metriccarpet}, Theorem \ref{topcarpet} then follows.
\end{proof}

As a preparation for the proof of Theorem~\ref{QScarpet}, we first introduce the definition of the SLE loop measure on $\C$, originally given in \cite{Zha21}.

Given two points $p,q\in\C$, we define $\SLE_\kappa^{p\rightleftharpoons q}$ to be be the law of the simple loop $\gamma\in\C$ obtained as follows. Let $\gamma_1$ be a whole-plane $\SLE_\kappa(2)$ curve from $p$ to $q$, and $\gamma_2$ be a conditionally independent chordal $\SLE_\kappa$ from $q$ to $p$ in $\C\backslash\gamma_1$. Then $\gamma$ is the concatenation of $\gamma_1$ and $\gamma_2$.

\begin{Definition}\label{Def:SLEloop}
    The \emph{$\SLE_\kappa$ loop measure} on $\C$ is an infinite measure of the space of simple loops on $\C$ defined as follows: For any measurable subset $A$ of collection of simple loops,
    \begin{equation}\label{eq:def-sleloop}
        \SLE_\kappa^{\rm loop}(d\gamma) = \frac{1}{\mathcal{M}(\gamma)^2} \iint_{\C\times\C}\frac{1}{|p-q|^{2(2-d)}}\SLE_\kappa^{p\rightleftharpoons q}(d\gamma) \ dp \ dq,
    \end{equation}
    where $d=1+\frac{\kappa}{8}$ and $\mathcal{M}(\gamma)$ is the $(1+\frac{\kappa}{8})$-dimensional Minkowski content of $\gamma$. See \cite{LR15} for more information on the Minkowski content of $\SLE$.
\end{Definition}

\begin{Proposition}\label{prop:non-qc}
    $\SLE_\kappa^{\rm loop}(\{\gamma: \gamma \ \textnormal{is a quasicircle on} \ \C \})=0$.
\end{Proposition}
\begin{proof}
    The proof is quite straightforward, following directly from a combination of Definition \ref{Def:SLEloop} and  Remark \ref{SLEgeneral}. 
    
    First notice that it is sufficient to show that $\SLE_\kappa^{p\rightleftharpoons q}(\{\gamma: \gamma {\rm\ is\ a\ quasicircle} \})=0$ for any $p,q\in\C$ by the definition of SLE loop measure. Let $\gamma$, $\gamma_1$ and $\gamma_2$ be curves as defined in the definition of $\SLE_\kappa^{p\rightleftharpoons q}$. Then $\gamma_1$ is not a quasiarc a.s. by Remark \ref{SLEgeneral}, which implies that $\gamma$ is a.s. not a quasicircle under the law $\SLE_\kappa^{p\rightleftharpoons q}$. The result then follows.
\end{proof}

The $\SLE_\kappa$ loop measure is closely related to the whole-plane version of $\CLE_\kappa$ defined in \cite[Section 3.2]{KW16}. To this end we need the notion of {\it nested} $\CLE_\kappa$. Namely, for a simply connected domain $D$, the nested $\CLE_\kappa$ on $D$ is defined from the (non-nested) $\CLE_\kappa$ by an iterative procedure: First sample a non-nested $\CLE_\kappa$ on $D$, which we refer to as the first generation. Then independently sample non-nested $\CLE_\kappa$s inside each loop in the first generation, resulting in the second generation. This procedure is repeated iteratively, continuing to infinity. The \emph{whole-plane $\CLE_\kappa$}, which we denote as $\CLE_\kappa^\C$, can be obtained as the weak limit of {\it nested} $\CLE_\kappa$ on $R\D$ as $R\to\infty$.

The following result establishes a relation between $\SLE_\kappa^{\rm loop}$ and the whole-plane $\CLE_\kappa^\C$.

\begin{Proposition}[{\cite[Theorem 1.1]{ACSW}}]\label{prop:equivalence}
	The $\SLE_\kappa^{\rm loop}$ equals the counting measure on the collection of loops in the whole-plane $\CLE_\kappa^\C$. More precisely, for any measurable subset $A$ of the space of simple loops on $\C$, we have
	\begin{equation}\label{eq:counting}
		\SLE_\kappa^{\rm loop}(A)=\int \left(\sum_{\gamma\in\Gamma}{\bf 1}_{\gamma \in A}\right)\CLE_\kappa^\C(d \Gamma),
	\end{equation}
	where $\Gamma$ stands for the configuration of $\CLE_\kappa^\C$.
\end{Proposition}

The following lemma establishes a relation between the whole-plane $\CLE_\kappa$ and $\CLE_\kappa$ on a simply connected domain. It is essentially the Markov property of $\CLE_\kappa$, and can be obtained immediately from the construction of whole-plane $\CLE_\kappa$ in \cite{KW16}. We denote by $D(\gamma)$ the bounded domain surrounded by $\gamma$ for a any simple loop $\gamma \sub \C$.

\begin{Lemma}\label{lem:markov}
    Let $\Gamma$ be the loop configuration of a whole-plane $\CLE_\kappa$ and $\widetilde\gamma$ be the outermost loop in $\{\gamma\in\Gamma:\gamma \sub \D\}$. Then given $\widetilde{\gamma}$ and the domain $D(\widetilde{\gamma})$ bounded by $\widetilde{\gamma}$, the conditional law of $\Gamma|_{D(\widetilde{\gamma})}=\{\gamma\in\Gamma:\gamma\sub D(\widetilde{\gamma})\}$ equals the independent nested $\CLE_\kappa$ on $D(\widetilde{\gamma})$.
\end{Lemma}

\begin{proof}
This is a straightforward result from \cite[Section 3.2]{KW16}, and here we briefly explain the idea. Let $\Gamma^{R}$ be sampled from the nested $\CLE_\kappa$ on ${R\D}$ (with $R>1$), and $\widetilde\gamma^{R}$ be the outermost loop in $\left\{\gamma\in\Gamma^{R}:\gamma \sub \D\right\}$. Then according to the construction of nested $\CLE_\kappa$, given $\widetilde{\gamma}^{R}$ and the domain $D(\widetilde{\gamma}^R)$ bounded by $\widetilde{\gamma}^R$, the conditional law of $\Gamma^{R}|_{D(\widetilde{\gamma}^R)}=\{\gamma\in\Gamma^{R}:\gamma\sub D(\widetilde{\gamma}^R)\}$ is the independent nested $\CLE_\kappa$ on $D(\widetilde{\gamma}^R)$. The lemma then follows by taking the weak limit as $R\to\infty$.
\end{proof}

According to Proposition \ref{prop:non-qc} and Lemma~\ref{lem:markov}, we can show that for any simply connected domain $D$ and $\CLE_\kappa$ on $D$, a.s. any loop in the $\CLE_\kappa$ configuration is not a quasicircle.

\begin{Theorem}\label{Theorem: nonquasiarc}
     Let $\Gamma$ be the loop configuration of $\CLE_\kappa$ on a simply connected domain $D$, then a.s. $\gamma$ is not a quasicircle for any $\gamma \in \Gamma$.
\end{Theorem}

\begin{proof}
    For any simply connected domain $D$, we denote by $\CLE_\kappa^D$ the law of $\CLE_\kappa$ on $D$. Let $A$ be the collection of all quasicircles on $\C$. Then by the conformal invariance of $\CLE_\kappa$ and Theorem \ref{compactQS}, we 
    \[
    \int \left(\sum_{\gamma\in\Gamma}{\bf 1}_{\gamma \in A}\right)\CLE_\kappa^D(d \Gamma) = C
    \]
    where $C$ is a constant independent of $D$.
    
   It follows Propositions~\ref{prop:non-qc} and~\ref{prop:equivalence} that $\int \left(\sum_{\gamma\in\Gamma}{\bf 1}_{\gamma \in A}\right)\CLE_\kappa^\C(d \Gamma) = 0$.  Recall the settings in Lemma~\ref{lem:markov}. Let $\nu(d\widetilde\gamma)$ be the law of the outermost loop in the configuration of whole-plane $\CLE_\kappa$ which surrounds the origin and is contained in $\D$.  Let $\CLE_{\kappa, {\rm nested}}^{D}$ be the law of nested $\CLE_\kappa$ on $D$. Since the nested CLE is defined by an iteration procedure from non-nested CLE, we have for all simply connected domain $D$,
    $$\int \left(\sum_{\gamma \in\Gamma}{\bf 1}_{\gamma \in  A}\right) \CLE_\kappa^{D}(d \Gamma) \le\int \left(\sum_{\gamma \in\Gamma}{\bf 1}_{\gamma \in  A}\right) \CLE_{\kappa, {\rm nested}}^{D}(d \Gamma).$$
    Then we obtain
    \begin{align*}
        \int \int \left(\sum_{\gamma \in\Gamma}{\bf 1}_{\gamma \in  A}\right) \CLE_\kappa^{D(\widetilde{\gamma})}(d \Gamma) \nu(d\widetilde{\gamma}) & \leq \int \int \left(\sum_{\gamma \in\Gamma}{\bf 1}_{\gamma \in  A}\right) \CLE_{\kappa,\ {\rm nested}}^{D(\widetilde{\gamma})}(d \Gamma) \nu(d\widetilde{\gamma})
        \\
        & \leq \int \left(\sum_{\gamma\in\Gamma}{\bf 1}_{\gamma \in A}\right)\CLE_\kappa^\C(d \Gamma) = 0.        
    \end{align*}
     while the last inequality follows from Lemma~\ref{lem:markov}. 
    
    In fact, we also have that
    \[
    \int \int \left(\sum_{\gamma \in\Gamma}{\bf 1}_{\gamma \in  A}\right) \CLE_\kappa^{D(\widetilde{\gamma})}(d \Gamma) \nu(d\widetilde{\gamma}) = C \geq 0
    \]
    since $\int \left(\sum_{\gamma \in\Gamma}{\bf 1}_{\gamma \in  A}\right) \CLE_\kappa^{D(\widetilde{\gamma})}$ equals $C$ for all $\widetilde{\gamma}$.
    
    Then we conclude
    \begin{equation}
        \int \left(\sum_{\gamma\in\Gamma}{\bf 1}_{\gamma \in A}\right)\CLE_\kappa^D(d \Gamma)=0,
    \end{equation}
    i.e. almost surely every loop in a $\CLE_\kappa$ configuration on $D$ is not a quasicircle.
\end{proof}

According to Theorem \ref{Theorem: nonquasiarc}, individual loops in the CLE configuration are not quasicircles a.s.. Thus proves the Theorem  \ref{QScarpet}.

\section{Further questions}\label{sec:FQ}

Inspired by the quasisymmetric uniformization projects of metric carpets \cite{Bon11, BM13, Mer10, MTW13, Hak22, HL23}, we are asking the following questions:

\textbf{Question 1.} What is the most reasonable quasisymmetric uniformizing space for random carpets?  The round carpet plays this role in the classical world of geometric group theory and complex dynamics. Theorem \ref{QScarpet} shows the stochastic setting requires a different model.

\textbf{Question 2.} Recall the Hausdorff dimension of $\CLE_\kappa$, $\kappa \in (\frac{8}{3}, 8]$ space is a.s. $2-\frac{(3\kappa-8)(8-\kappa)}{32\kappa}$ \cite{SSW09, SW11, MSW14}. Moreover, the Hausdorff dimension of one $\CLE_\kappa$ loop is a.s. $1+\frac{\kappa}{8}$ since any loop in a $\CLE_\kappa$ configuration is locally absolutely continuous with $\SLE_\kappa$. Can one lower these by applying a quasisymmetric homeomorphism? We define the \emph{conformal dimension} of a metric space be the infimum of the Hausdorff dimensions of all its quasisymmetric images. For example, the graph of one-dimensional Brownian motion has a.s. conformal dimension $\frac{3}{2}$, see \cite[Theorem~$1.1$]{BHL}.

We expect the conformal dimension of a $\CLE_\kappa$ space equals to its Hausdorff dimension a.s. for each $\kappa \in (\frac{8}{3}, 8]$. Intuitively, we expect that the $\CLE_\kappa$ space holds some product-like structure and such a structure is crucial in determine the conformal dimension of a metric space. See \cite{BHL}. We give a rough description of the product like structure. Let $\{S_t\}_ {t \in \R}$ be a collection of parallel lines in $\C$ such that it forms a disjoint cover of $\C$. Given a $\CLE_\kappa$ space $X$, we expect that a.s $\dim_H(X \cap S_t) = \dim_H(X) -1$ for a.e. $t \in \R$ when $S_t \cap X \neq \emptyset$. This is a natural phenomenon from the probabilistic point of view. For example, consider a critical Ising model with $+$ boundary condition, whose scaling limit is $\CLE_3$. The collection of sites that have a path to the boundary with all $+$ spins correspond to the $\CLE_3$ carpet. Then if we denote $p^\delta(z^\delta)$ to be the probability that $z^\delta$ is in the carpet, then we have $p^\delta(z^\delta)\sim \delta^{2-d}$ where $d=\dim_H(X)$. Therefore, the expected number in $S_t^\delta\cap X^\delta$ is similar to $\delta^{2-d}\cdot\delta^{-1}=\delta^{1-d}$. This suggests that the $\CLE_\kappa$ space is expected to exhibit a product-like structure.


\begin{thebibliography}{MTW12}

\bibitem[Ahl63]{Ahl63} L. Ahlfors, Quasiconformal reflections, Acta Math., \textbf{109}, (1963), 291--301.

\bibitem[Ahl06]{Ahl06} L. Ahlfors, Lectures on Quasiconformal Mapping, American mathematical Society, Providence, (2006).

\bibitem[ACSW]{ACSW} M. Ang, G. Cai, X. Sun and B. Wu, SLE Loop Measure and Liouville Quantum Gravity, preprint, arXiv: 2409.16547.

\bibitem[BE19]{BE19} A. Berlinkov and E. J\"arvenp\"a\"a, Porosities of Mandelbrot percolation, J. Theoret. Probab., \textbf{32}, (2019), 608--632.

\bibitem[BHL]{BHL} I. Binder, H. Hakobyan and W. Li, Conformal dimension of the Brownian graph, Duke Math. J., to appear.

\bibitem[BP17]{BP17} C. Bishop and Y. Peres, Fractals in Probability and Analysis, Cambridge Stud. Adv. Math. \textbf{162}, Cambridge University Press, (2017).

\bibitem[Bon06]{Bon06} M, Bonk, Quasiconformal geometry of fractals. International Congress of Mathematicians, \textbf{2}, Eur. Math. Soc., Z\"urich, (2006), 1349--1373.

\bibitem[Bon11]{Bon11} M. Bonk, Uniformization of Sierpi\'nski carpets in the plane, Invent. Math., \textbf{186}, (2011), 559--665.

\bibitem[BK02]{BK02} M. Bonk and B. Kleiner, Quasisymmetric parametrizations of two-dimensional metric spheres, Invent Math., \textbf{150}, (2002),127--183.

\bibitem[BK05]{BK05} M. Bonk and B. Kleiner, Conformal dimension and Gromov hyperbolic groups with $2$-sphere boundary, Geom. Topol., \textbf{9}, (2005), 219--246.

\bibitem[BKM09]{BKM09} M. Bonk, B. Kleiner and S. Merenkov, Rigidity of Schottky sets, Amer. J. Math., \textbf{131}, (2009), 409--443.

\bibitem[BLM16]{BLM16} M. Bonk, M. Lyubich and S Merenkov, Quasisymmetries of Sierpi\'nski carpet Julia sets, Adv. Math., \textbf{301}, (2016), 383--422.

\bibitem[BM13]{BM13} M. Bonk and S. Merenkov, Quasisymmetric rigidity of square Sierpi\'nski carpets, Ann. of Math., \textbf{177}, (2013), 591--643.

\bibitem[BP23]{BP23} N. Berestycki and E. Powell, Gaussian free field and Liouville quantum gravity, preprint, arXiv: 2404.16642.

\bibitem[Can94]{Can94} J. Cannon, The combinatorial Riemann mapping theorem, Acta Math., \textbf{173}, (1994),155--234.

\bibitem[CCD88]{CCD88}J. T. Chayes, L. Chayes and R. Durrett, Connectivity properties of Mandelbrot's percolation process, Probab. Theory Related Fields, \textbf{77}, (1988), 307--324.

\bibitem[CORS17]{CORS17} C. Chen, T. Ojala, E. Rossi and V. Suomala, Fractal percolation, porosity, and dimension, J. Theoret. Probab., \textbf{30}, (2017), 1471--1498.

\bibitem[DDDF20]{DDDF20} J. Ding, J. Dub\'edat, A. Dunlap, and H. Falconet, Tightness of Liouville first passage percolation for $\gamma \in (0, 2)$, Publ. Math. Inst. Hautes Etudes Sci., \textbf{132}, (2020), 353--403.
	
\bibitem[DS11]{DS11} B. Duplantier and S. Sheffield, Liouville quantum gravity and KPZ, Invent. Math., \textbf{185}, (2011), 333--393. 

\bibitem[GM21]{GM21} E. Gwynne and J. Miller, Existence and uniqueness of the Liouville quantum gravity metric for $\gamma \in (0,2)$, Invent. Math., \textbf{223}, (2021), 213--333. 

\bibitem[Hai15]{Hai15} P. Ha\"issinsky, Hyperbolic groups with planar boundaries, Invent. Math., \textbf{201}, (2015), 239--307.

\bibitem[Hak22]{Hak22} H. Hakobyan, Quasisymmetrically co-Hopfian Menger curves and Sierpi\'nski spaces, Adv.  Math., \textbf{402}, (2022), 108335.

\bibitem[HL23]{HL23} H, Hakobyan and W. Li, Quasisymmetric embeddings of slit Sierpi\'nski carpets, Trans. Amer. Math. Soc., \textbf{376}, (2023), 8877--8918.

\bibitem[Hei01]{Hei01} J, Heinonen, Lectures on Analysis on Metric Spaces, Springer-Verlag, New York, (2001).

\bibitem[Hug24]{Hug24}L. Hughes, Liouville quantum gravity metrics are not doubling, Electron. Commun. Probab., \textbf{29}, (2024), 1 -- 13.

\bibitem[Kem17]{Kem17} A. Kemppainen, Schramm–Loewner Evolution, SpringerBriefs Math. Phys., \textbf{24}, Springer, (2017).

\bibitem[KK00]{KK00} M, Kapovich and B. Kleiner, Hyperbolic groups with low-dimensional boundary, Ann. Sci. \'Ec. Norm. Sup\'er., \textbf{33}, (2000), 647--669.

\bibitem[KW16]{KW16} A. Kemppainen and W. Werner, The nested simple conformal loop ensembles in the Riemann sphere, Probab. Theory Relat. Fields, \textbf{165}, (2016), 835--866.

\bibitem[Kle06]{Kle06} B. Kleiner, The asymptotic geometry of negatively curved spaces: uniformization, geometrization and rigidity, International Congress of Mathematicians, \textbf{2}, Eur. Math. Soc., Z\"urich, (2006), 743--768.

\bibitem[Law05]{Law05} G. Lawler, Conformally Invariant Processes in the Plane, Mathematical Surveys and Monographs \textbf{114}, American Mathematical Society, (2005).

\bibitem[LR15]{LR15} G. Lawler and M. Rezaei, Minkowski content and natural parameterization for the Schramm–Loewner evolution. Annals of probability, \textbf{43}(3), (2015), 1082--1120.

\bibitem[LW04]{LW04} G. Lawler and W. Werner, The Brownian loop soup, Probab. Theory Related Fields, \textbf{128}, (2004), 565--588.

\bibitem[LSW03]{LSW03} G. Lawler, O. Schramm and W. Werner, Conformal restriction: the chordal case, J. Amer. Math. Soc., \textbf{16}, (2003), 917--955.

\bibitem[MTW13]{MTW13} J. Mackay, J. Tyson and K. Wildrick, Modulus and Poincar\'e inequalities on non-self-similar Sierpinski carpets,  Geom. Funct. Anal., \textbf{23} (2013), 985--1034.

\bibitem[Man74]{Man74}B. Mandelbrot, Intermittent turbulence in self-similar cascades- divergence of high moments and dimension of the carrier, J. Fluid Mech., \textbf{62}, (1974), 331--358.

\bibitem[Man83]{Man83} B. Mandelbrot, The fractal geometry of nature revised and enlarged edition, New York, WH Freeman and Co., (1983).

\bibitem[MR05]{MR05} D. Marshall and S. Rohde, The Loewner Differential Equation and Slit Mappings,  J. Amer. Math. Soc., \textbf{18}, 763--778, (2005).

\bibitem[Mer10]{Mer10} S. Merenkov, A Sierpi\'nski carpet with the co-Hopfian property, Invent. Math, \textbf{180}, (2010), 361--388.

\bibitem[MW11]{MW11} S Merenkov and K. Wildrick, Quasisymmetric Koebe uniformization., Rev. Mat. Iberoam., \textbf{29}, (2013), 859--909.

\bibitem[Mey10]{Mey10} D. Meyer, Snowballs are quasiballs, Trans. Amer. Math. Soc., \textbf{362}, (2010), 1247--1300.

\bibitem[MS16]{MS16} J. Miller and S. Sheffield, Imaginary geometry I: interacting SLEs, Probab. Theory Related Fields, \textbf{164}, (2016), 553--705. 

\bibitem[MS17]{MS17}  J. Miller and S. Sheffield, Imaginary Geometry IV: interior rays, whole-plane reversibility, and space-filling trees, Probab. Theory Related Fields, \textbf{169}, (2017), 729--869.

\bibitem[MSW14]{MSW14} J. Miller, N. Sun and D. Wilson, The Hausdorff dimension of the CLE gasket, Ann. Probab., \textbf{42}, (2014), 1644--1665.

\bibitem[MP10]{MP10} P. Morters and Y. Peres, Brownian Motion, Cambridge University Press, (2010).

\bibitem[RS05]{RS05} S. Rohde and O. Schramm, Basic properties of SLE, Ann. of Math., \textbf{161}, (2005), 883--924.

\bibitem[Sch00]{Sch00} O. Schramm, Scaling limits of loop-erased random walks and uniform spanning trees, Israel J. Math., \textbf{118}, (2000), 221--288.

\bibitem[SSW09]{SSW09} O. Schramm, S. Sheffield and D. Wilson, Conformal radii for conformal loop ensembles, Comm. Math. Phys., \textbf{288}, (2009), 43--53.

\bibitem[SW11]{SW11} N. Serban and W. Wendelin, Random soups, carpets and fractal dimensions, J. Lond. Math. Soc., \textbf{83} (2011), 789--809.

\bibitem[She09]{She09} S. Sheffield, Exploration trees and conformal loop ensembles, Duke Mathematical Journal, \textbf{147}, 79--129, (2009).

\bibitem[SW12]{SW12} S. Sheffield and W. Werner, Conformal loop ensembles: the Markovian characterization and the loop-soup construction, Ann. of Math., \textbf{176}, (2012), 1827--1917.

\bibitem[Sul83]{Sul83} D. Sullivan, Conformal dynamical systems. In Geometric dynamics, Lecture Notes in Math.,Springer, Berlin, \textbf{1007}, (1983), 725--752.

\bibitem[Sul85a]{Sul85a} D. Sullivan, Quasiconformal homeomorphisms and dynamics I. Solution of the Fatou-Julia problem on wandering domains, Ann. of Math., \textbf{122}, (1985), 401--418.

\bibitem[Sul85b]{Sul85b} D. Sullivan, Quasiconformal homeomorphisms and dynamics II. Structural Stability Implies Hyperbolocity for Kleinian Groups, Acta Math., \textbf{155}, (1985), 243--260.

\bibitem[Tor21]{Tor21} S Troscheit, On quasisymmetric embeddings of the Brownian map and continuum trees, Probab. Theory Related Fields, \textbf{179}, (2021), 1023--1046.

\bibitem[TV80]{TV80} P. Tukia and J. V\"ais\"al\"a, Quasisymmetric Embeddings of Metric Spaces, Ann. Acad. Sci. Fenn. Math., \textbf{5}, (1980), 97--114.

\bibitem[Why58]{Why58} G. Whyburn, Topological characterization of the Sierpi\'nski curve, Fund Math., \textbf{45}, (1958), 320--340.

\bibitem[Wil08]{Wil08} K. Wildrick, Quasisymmetric parametrizations of two-dimensionalmetric planes, Proc. London Math. Soc., \textbf{97}, (2008), 783--812.

\bibitem[Wu15]{Wu15} H. Wu, Conformal restriction: The radial case, Stochastic Process. Appl., \textbf{125}, (2015), 552-570.

\bibitem[Zha15]{Zha15} D. Zhan, Reversibility of whole-plane SLE, Probab. Theory Related Fields, \textbf{161}, (2015), 561--618.

\bibitem[Zha21]{Zha21} D. Zhan, SLE loop measures, Probab. Theory Related Fields, \textbf{179}, (2021), 345--406.
\end{thebibliography}
\end{document}